\documentclass[12pt, reqno, a4paper]{amsart}

\usepackage{ amssymb, amsmath, enumerate, amsfonts, amsthm, mathrsfs, url, bm, mathtools}
\usepackage{placeins}

\usepackage{xcolor}  	
\usepackage{hyperref}
\hypersetup{
colorlinks,
   linkcolor={cyan!80!black},
   citecolor={cyan!80!black},
 urlcolor={cyan!80!black}
}

\usepackage{color}

\usepackage[margin=1in]{geometry}

\RequirePackage{doi}

\usepackage{amscd}
\usepackage{amsfonts}
\usepackage{float}
\usepackage{color}
\usepackage[
backend=biber,
style=alphabetic,
]{biblatex}
\usepackage{bookmark}

\renewbibmacro{in:}{}
\DeclareFieldFormat{title}{#1}

\DeclareFieldFormat[article]{title}{\mkbibemph{#1}}       % Articles
\DeclareFieldFormat[incollection]{title}{\mkbibemph{#1}}  % Incollection chapters
\DeclareFieldFormat[book]{title}{\mkbibemph{#1}}          % Standalone books
\DeclareFieldFormat[incollection]{booktitle}{#1}          % Plain for incollection book titles
\DeclareFieldFormat[article]{journaltitle}{#1}            % Plain for journals
\DeclareFieldFormat[proceedings]{title}{\mkbibemph{#1}}   % Italicize proceedings titles
\DeclareFieldFormat[proceedings]{booktitle}{\mkbibemph{#1}} % Italicize book titles in proceedings

\AtEveryBibitem{%
  \ifentrytype{misc}{\DeclareFieldFormat{title}{\mkbibemph{#1}}}{}}

\DeclareFieldFormat{eprint:eprint}{arXiv:\href{https://arxiv.org/abs/#1}{#1}}

\usepackage{amssymb}
\theoremstyle{plain}
\newtheorem{theorem}{Theorem}[section]
\newtheorem{lemma}{Lemma}[section]
\newtheorem{corollary}{Corollary}[section]
\newtheorem{proposition}{Proposition}[section]

\theoremstyle{definition}
\newtheorem{definition}{Definition}[section]

\newtheorem{conjecture}{Conjecture}[section]

\theoremstyle{remark}
\newtheorem{remark}{Remark}[section]

\numberwithin{equation}{section}

\newcommand{\Mod}[1]{\ (\mathrm{mod}\ #1)}

\renewcommand{\leq}{\leqslant}
\renewcommand{\geq}{\geqslant}

\DeclareRobustCommand{\stirling}{\genfrac\{\}{0pt}{}}

\begin{document}

\title{Pseudorandomness of primes at large scales}


%    author one information
\author{}
\address{}
\curraddr{}
\email{}
\thanks{}

%    author two information
\author{Sun-Kai Leung}
\address{D\'epartement de math\'ematiques et de statistique\\
Universit\'e de Montr\'eal\\
CP 6128 succ. Centre-Ville\\
Montr\'eal, QC H3C 3J7\\
Canada}
\curraddr{}
\email{sun.kai.leung@umontreal.ca}
\thanks{}

  %  \subjclass is required.

\date{}

\dedicatory{}


%    Abstract is required.
\begin{abstract}

Assuming a $q$-variant of the prime $k$-tuple conjecture uniformly, we compute mixed moments of the number of primes in disjoint short intervals and progressions, respectively. This involves estimating the mean of singular series along products of lattices, which is of independent interest. As a consequence, we establish the convergence of both sequences of suitably normalized primes to a standard Poisson point process. 

\end{abstract}

\maketitle

%\setcounter{tocdepth}{1}
%\tableofcontents

\section{Introduction}

The study of prime numbers dates back over 2000 years to the era of Euclid and Eratosthenes. While significant progress has been made during these two millennia, the exact distribution of primes remains a mystery. One of the biggest challenges is that, despite being a deterministic sequence of natural numbers, primes exhibit highly pseudorandom behavior, apart from some obvious ``local obstructions" (see \cite[p. 194]{MR2459552}), such as the absence of consecutive primes other than the pair $\{2,3\}$ due to parity. Therefore, natural questions arise: to what extent are primes pseudorandom at appropriate larger scales? Can such pseudorandomness be rigorously justified?\footnote{See also, for instance, \cite[Section 9]{MR2415379} for the notion of pseudorandomness in the context of Green's transference principle.}

In this paper, we provide one possible answer that suitably normalized primes in short intervals and short progressions, respectively, are located on the positive real number line $(0,\infty)$ as if they are randomly positioned points in the sense of a \textit{Poisson point process}.

Such a point process models the random and independent occurrence of events. Due to its Poissonian property and complete independence, a Poisson point process is often referred to as a purely random process. Thanks to the Kolmogorov extension theorem, it is completely characterized by the finite-dimensional distribution, i.e., the joint probability distributions for the number of events in all finite collections of disjoint intervals (see \cite[p. 19]{MR1950431} for instance). This leads to the following convenient definition.

\begin{definition}[Standard Poisson point process] \label{def:poisson}
A point process $\xi=\sum_{n=1}^{\infty}\delta_{X_n}$ (or $\{ X_n \}_{n = 1}^{\infty}$ by abuse of notation) on $(0,\infty)$ is called a \textit{standard Poisson point process} if 
%\begin{itemize}
   % \item $N(\xi,\{x\})=0$ for any $x \in (0,\infty)$ almost surely;
for any integer $r \geq 1$ and collection of disjoint intervals $\{I_i\}_{i=1}^{r}$ in $(0,\infty),$ we have
\begin{gather*}
\mathbb{P}\left( \xi(I_i)=k_i, 1 \leq i \leq r \right)
=\prod_{i=1}^{r} e^{-|I_i|}\frac{|I_i|^{k_i}}{{k_i}!},
\end{gather*}
%\end{itemize}
where $ \xi(A):=\#\{ n \geq 1 \,:\, X_n \in A \}$ is the number of points in the subset $A \subseteq (0,\infty)$.
\end{definition}

To be more precise, we shall establish the convergence of sequences of suitably normalized primes to a standard Poisson point process conditionally. For our purpose, we first recall the definition of convergence in distribution of point processes (see \cite[pp. 3--4, 109]{MR3642325} and \cite[Theorem 4.11]{MR3642325} for instance), and propose a $q$-variant of the prime $k$-tuple conjecture, which is essentially \cite[Conjecture 1.1]{moments} without the power-saving error term, again inspired by the modified Cram\'{e}r model (see \cite[pp. 23--24]{MR1349149}).

\begin{definition}[Convergence in distribution] \label{def:conv}
Let $\xi, \xi_1, \xi_2, \ldots$ be point processes on $(0,\infty).$ Then we say \textit{$\xi_n$ converges to $\xi$ in distribution as $n \to \infty$} if $\xi_n \xrightarrow[]{vd} \xi$ as random measures. Equivalently, for any integer $r \geq 1$ and any collection of disjoint $\xi$-continuity intervals $\{I_i\}_{i=1}^{r}$ in $(0,\infty),$\footnote{If $\xi$ is a standard Poisson point process, then all intervals $I \subseteq (0,\infty)$ are automatically $\xi$-continuity sets, i.e., $\xi(\partial I)=0$ almost surely.} we have the convergence in distribution
\begin{align*}
(\xi_n(I_1), \ldots, \xi_n(I_r)) \xrightarrow[]{d} (\xi(I_1), \ldots, \xi(I_r))
\end{align*}
as $n \to \infty.$
\end{definition}

\begin{conjecture}[$q$-variant of Hardy--Littlewood prime $k$-tuple conjecture (HL\text{[$q$]})] \label{conj:hl}
Given integers $k,q, H \geq 1,$ let $N \gg \varphi(q)\log q$ be an integer and $\mathcal{H} \subseteq [0,H]$ be a set consisting of $k$ distinct multiples of $q$ which is \textit{admissible}, i.e., $v_p(\mathcal{H}):=\#\{ h \in \mathcal{H} \Mod{p} \}<p$ for any prime $p.$ Then as $N \to \infty,$ we have
\begin{align*}
\sum_{\substack{1 \leq n \leq N\\(n,q)=1}} \prod_{h \in \mathcal{H}}1_{\mathbb{P}}(n+h) = (1+o_{k,H}(1))
\mathfrak{S}(\mathcal{H};q)
\sum_{\substack{1 \leq n \leq N\\(n,q)=1}} 
\left( \prod_{h \in \mathcal{H}} \log (n+h) \right)^{-1},
\end{align*}
where $1_{\mathbb{P}}$ is the prime indicator function, and $\mathfrak{S}(\mathcal{H};q)$ is the \textit{singular series}
defined as the Euler product
\begin{align*}
\left( \frac{\varphi(q)}{q} \right)^{-k}
\prod_{p \nmid q} \left( 1-\frac{1}{p} \right)^{-k}
\left( 1-\frac{v_p(\mathcal{H})}{p} \right).
\end{align*}
\end{conjecture}

Note that $
\mathfrak{S}(\mathcal{H};q) = \left( \varphi(q)/q \right)^{-1} \mathfrak{S}(\mathcal{H}),$ where 
\begin{align*}
\mathfrak{S}(\mathcal{H}):=\prod_{p} \left( 1-\frac{1}{p} \right)^{-k}
\left( 1-\frac{v_p(\mathcal{H})}{p} \right)
\end{align*}
is the classical singular series.

Assuming HL[1] with sufficient uniformity, Gallagher \cite{MR409385} showed that the prime count in short intervals has a Poissonian limiting distribution. More precisely, if $\lambda>0,$ then for any integer $k \geq 0,$ we have
\begin{gather*}
\lim_{N \to \infty}
\frac{1}{N}\#\{ N <n \leq 2N : \pi(n+H)-\pi(n)=k \}
=e^{-\lambda} \frac{\lambda^k}{k!},
\end{gather*}
where $H=\lambda \log N$, i.e., we have the convergence in distribution to a Poissonian random variable with parameter $\lambda$
\begin{gather*}
\pi(n+H)-\pi(n) \xrightarrow[]{d} Poisson(\lambda).
\end{gather*}
For longer intervals, assuming HL[1] with a power-saving error term, Montgomery and Soundararajan \cite{MR2104891} showed that the prime count in short intervals has a Gaussian limiting distribution (see \cite[pp. 59--73]{MR2290490} for further discussion).

Inspired by their work, assuming HL[$q$] with a power-saving error term, the author \cite{moments} recently computed moments of the number of primes not exceeding $N$ in progressions to a common large modulus $q$ as $a \pmod{q}$ varies. Consequently, depending on the size of $\varphi(q)$ with respect to $N$, the prime count exhibits a Gaussian or Poissonian law. In particular, if $\lambda>0,$ then for any integer $k \geq 0,$ we have
\begin{align*}
\lim_{N \to \infty} \frac{1}{\varphi(q)}\left| \left\{ a \Mod{q} \,: \, 
\pi(N;q,a)=k \right\}
 \right| 
=e^{-\lambda}\frac{\lambda^k}{k!},
\end{align*}
where $N=\lambda\varphi(q)\log q$, i.e., we have the convergence in distribution to a Poissonian random variable with parameter $\lambda$
\begin{gather*}
\pi(N;q,a) \xrightarrow[]{d} Poisson(\lambda).
\end{gather*}
See also \cite{MR1024571} in which $a$ is fixed but $q$ varies instead.

As we shall demonstrate, both results of Gallagher and the author are special cases of Corollaries \ref{thm:shortint} and \ref{thm:ap} respectively.
\\~\\
\noindent\textit{Notation.} 
Throughout the paper, we use the standard big $O,$ little $o$ notations, and the Vinogradov notation $\ll,$ where the implied constants depend only on the subscripted parameters. We denote $\sum_{a \Mod{q}}^{*}$ as the sum over reduced residues modulo $q.$ We also write $\sum_{x_i}$ for $\sum_{x_1, \ldots, x_r}$ to simplify notation, unless otherwise specified.

\section{Main results}

Analogous to \cite{MR409385}, \cite{MR2104891} and \cite{moments}, to establish the convergence in distribution of point processes, we compute mixed moments of the number of primes in short intervals and short progressions respectively. To state our main results, let us denote $\operatorname{diam} (\mathcal{I}):= \sup \cup_{1 \leq i \leq r} I_i- \inf \cup_{1 \leq i \leq r} I_i$ as the diameter of the collection of disjoint intervals $\mathcal{I}=\{I_i\}_{i=1}^{r}$ in $(0,\infty),$ and
 $\stirling{k}{j}$ as the \textit{Stirling number of the second kind}, i.e., the number of ways to partition a set of $k$ objects into $j$ non-empty subsets.

\begin{theorem}
Given an integer $r \geq 1,$ let $\mathcal{I}=\{I_i\}_{i=1}^{r}$ be a collection of disjoint intervals in $(0,\infty).$ Suppose HL[1] holds for $1 \leq H \leq \operatorname{diam} (\mathcal{I})$ uniformly. Then for all integers $k_1, \ldots, k_r \geq 0,$ we have
\begin{gather*}
\lim_{N \to \infty} \frac{1}{N}\sum_{N<n \leq 2N}
\prod_{i=1}^r \#\left\{ p > n \,:\, 
\frac{p-n}{\log N} \in I_i
\right\}^{k_i}
=\prod_{i=1}^r \Bigg( \sum_{j_i=1}^{k_i} \stirling{k_i}{j_i}|I_i|^{j_i} \Bigg).
\end{gather*}

\end{theorem}

\begin{theorem}

Given an integer $r \geq 1,$ let $\mathcal{I}=\{I_i\}_{i=1}^{r}$ be a collection of disjoint intervals in $(0,\infty).$ Suppose HL[q] holds for $1 \leq H \leq \operatorname{diam} (\mathcal{I}) \varphi(q)\log q$ uniformly. Then for all integers $k_1, \ldots, k_r \geq 0,$ we have
\begin{gather*}
\lim_{q \to \infty}\frac{1}{\varphi(q)}\sideset{}{^*}\sum_{a \Mod{q}}
\prod_{i=1}^{r}  \#\left\{ p \equiv a \Mod{q} \,:\, 
\frac{p}{\varphi(q)\log q} \in I_i
\right\}^{k_i}
=\prod_{i=1}^r \Bigg( \sum_{j_i=1}^{k_i} \stirling{k_i}{j_i}|I_i|^{j_i}  \Bigg).
\end{gather*}

\end{theorem}

\begin{remark}
It is, in fact, easier and considered more natural to compute factorial moments rather than ordinary moments in this context. However, we opt for the latter to ensure compatibility with the existing literature.
\end{remark}

Applying the method of moments, by Definitions \ref{def:poisson} and \ref{def:conv}, the convergence of both sequences of suitably normalized primes in short intervals and short progressions to a standard Poisson point process is an immediate consequence.

\begin{corollary} \label{thm:shortint}

Given an integer $r \geq 1,$ let $\mathcal{I}=\{I_i\}_{i=1}^{r}$ be a collection of disjoint intervals in $(0,\infty).$ Suppose HL[1] holds for $1 \leq H \leq  \operatorname{diam} (\mathcal{I})$ uniformly. Then for all integers $k_1, \ldots, k_r \geq 0,$ we have
\begin{gather*}
\lim_{N \to \infty}\frac{1}{N} \# \left\{ N<n \leq 2N  \,:\, \#\left\{ p > n \,:\, 
\frac{p-n}{\log N} \in I_i
\right\}=k_i, 1 \leq i \leq r\right\}\\
=\prod_{i=1}^r e^{-|I_i|} \frac{|I_i|^{k_i}}{k_i!},
\end{gather*}
i.e., the point process $\{ \frac{p_i(n)-n}{\log N} \}_{i=1}^{\infty},$ for $N<n \leq 2N$ chosen uniformly at random, converges in distribution to the standard Poisson point process as $N \to \infty,$ where $p_i$ denotes the $i$-th least prime after $n$. In particular, by taking $r=1$ and $I_1=(0, \lambda],$ we recover \cite[Theorem 1]{MR409385}.
\end{corollary}

\begin{remark}
On the ``frequency side", one formulation of the GUE hypothesis states that the point process
$\{ \frac{\log T}{2\pi}(\gamma_n-t) \}$, for $T <t \leq 2T$ chosen uniformly at random, converges in distribution to the sine-kernel process, where $\gamma_n$ is the $n$-th (counted with potential multiplicities) least positive ordinate of a nontrivial zero of the Riemann zeta function $\zeta(s)$ (see \cite{dereyna2023convergence} for instance).
\end{remark}

\begin{corollary} \label{thm:ap}

Given an integer $r \geq 1,$ let $\mathcal{I}=\{I_i\}_{i=1}^{r}$ be a collection of disjoint intervals in $(0,\infty).$ Suppose HL[$q$] holds for $1 \leq H \leq  \operatorname{diam} (\mathcal{I})\varphi(q)\log q$ uniformly. Then for all integers $k_1, \ldots, k_r \geq 0,$ we have
\begin{gather*}
\lim_{q \to \infty}\frac{1}{\varphi(q)} \# \left\{ a \Mod{q} \,:\, \#\left\{ p \equiv a \Mod{q} \,:\, 
\frac{p}{\varphi(q)\log q} \in I_i
\right\}=k_i,  1 \leq i \leq r\right\}\\
=\prod_{i=1}^r e^{-|I_i|} \frac{|I_i|^{k_i}}{k_i!},
\end{gather*}
i.e., the point process $\{ \frac{p_i(q,a)}{\varphi(q)\log q} \}_{i=1}^{\infty},$ for $a \Mod{q}$ chosen uniformly at random, converges in distribution to the standard Poisson point process as $q \to \infty,$ where $p_i(q,a)$ denotes the $i$-th least prime which is congruent to $a \Mod{q}.$
In particular, by taking $r=1$ and $I_1=(0, \lambda],$ we recover \cite[Corollary 1.2]{moments}.
\end{corollary}

\begin{remark}
Similarly, it is also believed that
$\{ \frac{\log q}{2\pi}\gamma_{\chi,n} \},$ for a (non-trivial) primitive character $\chi \Mod{q}$ chosen uniformly at random, converges in distribution to the sine-kernel process, where $\gamma_{\chi,n}$ is the $n$-th (counted with potential multiplicities) least positive ordinate of a nontrivial zero of the Dirichlet L-function $L(s,\chi)$ (see \cite{MR2013141} for instance).
\end{remark}

The proofs of Theorems \ref{thm:shortint} and \ref{thm:ap} share the same structure, although the latter is more technical due to additional changes of variables. To avoid repetition, this paper includes only the proof of Theorem \ref{thm:ap}. 
%The detailed proof of Theorem \ref{thm:shortint} will be provided in the author's doctoral thesis.

\section{Estimates of singular series}

As observed in \cite{MR409385}, while primes exhibit ``local" correlations, leading to those singular series defined in Conjecture \ref{conj:hl}, such correlations diminish ``globally" as the classical singular series is $1$ on average over hypercubes. Similarly, the following (unconditional) %lemma by the author \cite{moments} 
proposition, which is a ``multidimensional" version of \cite[Lemma 6.1]{moments} but without an explicit error term, lies at the crux of the mixed moment computation.

\begin{proposition} \label{prop:sing}
%Let $k \geq 1$ and $(a,q)=1.$ Then for any $N \geq 2q,$ we have
%\begin{align*}
%\sum_{\substack{1 \leq n_i\leq N\\n_i \equiv a \Mod{q}\\ n_i \text{ \normalfont distinct}}}
%\mathfrak{S}(\mathcal{N};q)=\left( \frac{N}{\varphi(q)} \right)^k
%+O_k \left( \left( \frac{N}{\varphi(q)} \right)^{k-1} 
%\log \frac{N}{q}  \right),
%\end{align*}
Given integers $q,r \geq 1,$ let $\{ J_i \}_{i=1}^r$ be a collection of disjoint intervals in $(0,\infty)$ with $\min_{1 \leq i \leq r} |J_i|/q \to \infty$ as $\min_{1 \leq i \leq r} |J_i| \to \infty.$ Then for any $(b,q)=1$ and $l_1, \ldots, l_r \geq 1,$ we have
\begin{align*}
\underset{\substack{ h_1^{(i)}, \ldots, h_{l_i}^{(i)} \in J_i \text{ \normalfont distinct}\\ h_1^{(i)} \equiv \cdots \equiv h_{l_i}^{(i)} \equiv b \Mod{q} }}{\sum \cdots \sum} \mathfrak{S}(\mathcal{H};q)=(1+o_{l_1,\ldots,l_r}(1))\prod_{i=1}^r \left(\frac{|J_i|}{\varphi(q)} \right)^{l_i}
\end{align*}
as $\min_{1 \leq i \leq r} |J_i| \to \infty,$ where $\mathcal{H}=\{ h_1^{(i)}, \ldots, h_{l_i}^{(i)} \}_{i=1}^r,$ i.e., the singular series $\mathfrak{S}(\mathcal{H};q)$ is, on average, $(\varphi(q)/q)^{-(l_1+\cdots+l_r)}$  along a product of lattices of ranks $l_1, \ldots, l_r.$
\end{proposition}

\begin{proof}
We shall follow the proof in \cite{MR409385} (see \cite{ford2016simpleproofgallagherssingular} for a simplification). Let $p \nmid q$ be a prime. Then we define
\begin{align*}
\left( 1-\frac{1}{p} \right)^{-k}
\left( 1-\frac{v_p(\mathcal{H})}{p} \right)
=:1+a(p,v_p(\mathcal{H})),
\end{align*}
and
\begin{align*}
a(d;\mathcal{H}):=\prod_{p \mid d} a(p,v_p(\mathcal{H}))
\end{align*}
for any square-free integer $d$ coprime to $q.$ It follows from the definition that
\begin{align*}
\mathfrak{S}(\mathcal{H};q)=\sum_{\substack{d=1\\(d,q)=1}}^{\infty} 
\mu^2(d) a(d;\mathcal{H}),
\end{align*}
which can be verified to be absolutely convergent, so that
\begin{align} \label{eq:latticesum}
\underset{\substack{ h_1^{(i)}, \ldots, h_{l_i}^{(i)} \in J_i \text{ \normalfont distinct}\\ h_1^{(i)} \equiv \cdots \equiv h_{l_i}^{(i)} \equiv b \Mod{q} }}{\sum \cdots \sum} \mathfrak{S}(\mathcal{H};q) &=
\left( \frac{\varphi(q)}{q} \right)^{-|\boldsymbol{l}|}
\sum_{\substack{d=1\\(d,q)=1}}^{\infty} \mu^2(d) 
\underset{\substack{ h_1^{(i)}, \ldots, h_{l_i}^{(i)} \in J_i \text{ \normalfont distinct}\\ h_1^{(i)} \equiv \cdots \equiv h_{l_i}^{(i)} \equiv b \Mod{q} }}{\sum \cdots \sum} a(d;\mathcal{H}) \nonumber \\
&=:\left( \frac{\varphi(q)}{q} \right)^{-|\boldsymbol{l}|}
\sum_{\substack{d=1\\(d,q)=1}}^{\infty} \mu^2(d) S'(d),
\end{align}
where $|\boldsymbol{l}|:=l_1+\cdots+l_r.$ Arguing as in \cite{MR409385}, one can 
%choose a suitable integer $D \geq 1$ for which the contribution from $d>D$ is negligible. 
show that
\begin{align} \label{eq:d>D}
\sum_{\substack{d>D\\(d,q)=1}} \mu^2(d) S'(d)
\ll_{\boldsymbol{l},\epsilon}
(D|J|)^{\epsilon} \frac{|J|}{q^r D},
\end{align}
where $J:=\prod_{i=1}^r J_i.$
From now on, we always assume $1 \leq d \leq D.$ Interchanging the order of summation, the sum $S'(d)$ becomes
%\begin{align*}
%\left( \frac{\varphi(q)}{q} \right)^{-|\boldsymbol{l}|}
%\sum_{\substack{d \leq D\\(d,q)=1}} \mu^2(d) 
%\underset{\substack{ h_1^{(i)}, \ldots, h_{l_i}^{(i)} \in J_i \text{ \normalfont distinct}\\ h_1^{(i)} \equiv \cdots \equiv h_{l_i}^{(i)} \equiv a \Mod{q} }}{\sum \cdots \sum} a(d;\mathcal{H}) 
%\end{align*}
\begin{gather*}
\sum_{\boldsymbol{v}=(v_p)_{p \mid d}} \prod_{p|d} a(p,v_p)
\underset{\substack{ h_1^{(i)}, \ldots, h_{l_i}^{(i)} \in J_i \text{ \normalfont distinct}\\ h_1^{(i)} \equiv \cdots \equiv h_{l_i}^{(i)} \equiv b \Mod{q} \\ (v_p(\mathcal{H}))_{p \mid d} = \boldsymbol{v} }}{\sum \cdots \sum} 1
=:S(d)-R(d),
\end{gather*}
where $\boldsymbol{v}=(v_p)_{p \mid d} \in \mathbb{N}^{w(d)}$ with $w(d):=\#\{ p \,:\, p | d\},$
\begin{align*}
S(d):=\sum_{\boldsymbol{v}=(v_p)_{p \mid d} } \prod_{p|d} a(p,v_p)
\underset{\substack{ h_1^{(i)}, \ldots, h_{l_i}^{(i)} \in J_i \\ h_1^{(i)} \equiv \cdots \equiv h_{l_i}^{(i)} \equiv b \Mod{q} \\ (v_p(\mathcal{H}))_{p \mid d} = \boldsymbol{v} }}{\sum \cdots \sum} 1,
\end{align*}
and $R(d)$ is the contribution of all non-distinct terms to $S(d)$, which is
\begin{align*}
\ll_r \sum_{\boldsymbol{v}=(v_p)_{p \mid d}} \prod_{p|d} |a(p,v_p)| \left( \min_{1 \leq i \leq r} \frac{|J_i|}{q} \right)^{-1} \prod_{i=1}^r \left( \frac{|J_i|}{q} \right)^{l_i}.
\end{align*}
Applying the Chinese remainder theorem, the inner sum of $S(d)$ is
\begin{gather*}
\prod_{i=1}^r \left( \frac{|J_i|}{qd} +O(1) \right)^{l_i}
\prod_{p \mid d} {p \choose v_p} \stirling{|\boldsymbol{l}|}{v_p} v_p! \\
= \left( 1+O_r \left( \left( \min_{1 \leq i \leq r} \frac{|J_i|}{qd} \right)^{-1} \right) \right)\prod_{i=1}^r \left( \frac{|J_i|}{qd} \right)^{l_i}
\prod_{p \mid d} {p \choose v_p} \stirling{|\boldsymbol{l}|}{v_p} v_p!,
\end{gather*}
%\begin{align*}
%\prod_{i=1}^r \left( \frac{|J_i|}{\varphi(q)} + O \bigg( \left(\frac{\varphi(q)}{q}\right)^{-1} \bigg) \right)^{l_i} A(d)
%\end{align*}
so that
\begin{align} \label{eq:sd}
S'(d)=\prod_{i=1}^r \left( \frac{|J_i|}{qd} \right)^{l_i} A(d) +&
O_r \left( \left( \min_{1 \leq i \leq r} \frac{|J_i|}{qd} \right)^{-1} \prod_{i=1}^r \left( \frac{|J_i|}{qd} \right)^{l_i} B(d) \right) \nonumber\\
+& O_r  \left( \left( \min_{1 \leq i \leq r} \frac{|J_i|}{q} \right)^{-1} \prod_{i=1}^r \left( \frac{|J_i|}{q} \right)^{l_i} C(d) \right),
\end{align}
where
\begin{align*}
A(d):=&\sum_{\boldsymbol{v}} \prod_{p \mid d} a(p,v_p) 
{p \choose v_p} \stirling{|\boldsymbol{l}|}{v_p} v_p!
=\prod_{p \mid d} \sum_{v_p=1}^p a(p,v_p) 
{p \choose v_p} \stirling{|\boldsymbol{l}|}{v_p} v_p!
, \\
B(d):=&\sum_{\boldsymbol{v}} \prod_{p \mid d} |a(p,v_p)|
{p \choose v_p} \stirling{|\boldsymbol{l}|}{v_p} v_p!
=\prod_{p \mid d} \sum_{v_p=1}^p |a(p,v_p)| 
{p \choose v_p} \stirling{|\boldsymbol{l}|}{v_p} v_p!
\end{align*}
and
\begin{align*}
C(d):=\sum_{\boldsymbol{v}} \prod_{p \mid d} |a(p,v_p)|
=\prod_{p \mid d} \sum_{v_p=1}^p |a(p,v_p)|.
\end{align*}
Arguing as in \cite{MR409385}, the contribution of the big $O$ terms to (\ref{eq:latticesum}) is 
\begin{align} \label{eq:d<D}
\ll_{\epsilon} D^{1+\epsilon} \left( \min_{1 \leq i \leq r} \frac{|J_i|}{q} \right)^{-1} \prod_{i=1}^r \left( \frac{|J_i|}{q} \right)^{l_i}.
\end{align}

Finally, by definition the $p$-th factor $A(p)$ is
\begin{align*}
\frac{1}{(p-1)^{|\boldsymbol{l}|}}
\left( (p^{|\boldsymbol{l}|}-(p-1)^{|\boldsymbol{l}|}) \sum_{v=1}^p {p \choose v} \stirling{|\boldsymbol{l}|}{v} v! - 
p^{|\boldsymbol{l}|-1}  \sum_{v=1}^p v{p \choose v} \stirling{|\boldsymbol{l}|}{v} v!
\right).
\end{align*}
Using the identities 
\begin{align*}
\sum_{v=1}^p {p \choose v} \stirling{|\boldsymbol{l}|}{v} v!
= p^{|\boldsymbol{l}|}
\end{align*}
and
\begin{align*}
\sum_{v=1}^p v{p \choose v} \stirling{|\boldsymbol{l}|}{v} v!
=  (p^{|\boldsymbol{l}|}-(p-1)^{|\boldsymbol{l}|})p
\end{align*}
(see \cite[p. 8]{MR409385}), we conclude that $A(d)=1$ if $d=1$ and vanishes otherwise. Therefore, by taking 
\begin{align*}
D=\left(\min_{1 \leq i \leq r} \frac{|J_i|}{q}\right)^{1/2},
\end{align*}
the proposition follows from (\ref{eq:latticesum}), (\ref{eq:d>D}), (\ref{eq:sd}) and (\ref{eq:d<D}).
\end{proof}
We also require an upper bound for individual (classical) singular series.

\begin{lemma} \label{lem:upperbound}
Given integers $k,q\geq 1,$ let $H \geq 2q$ be an integer and $\mathcal{H} \subseteq [0,H]$ be a set consisting of $k$ distinct multiples of $q$. Then
\begin{align*}
\mathfrak{S}(\mathcal{H}) \ll_k \left(\log \frac{H}{q} \right)^{k-1}.
\end{align*}
\end{lemma}

\begin{proof}
By definition, we have
\begin{align} \label{eq:prod}
\mathfrak{S}\left( \mathcal{H} \right)
=
\prod_{\substack{p \leq H/q \\ p \nmid q}} \left( 1-\frac{1}{p} \right)^{-k}
\left( 1-\frac{v_p(\mathcal{H})}{p} \right)
\prod_{\substack{p > H/q \\ p \nmid q}} \left( 1-\frac{1}{p} \right)^{-k}
\left( 1-\frac{v_p(\mathcal{H})}{p} \right).
\end{align}
Since the set $\mathcal{H}$ consists of distinct multiples of $q$, we must have $v_p(\mathcal{H})=k$ provided that $p>H/q$ and $p \nmid q,$ so that the last product is 
\begin{align*}
\prod_{\substack{p > H/q \\ p \nmid q}} \left( 1-\frac{1}{p} \right)^{-k}
\left( 1-\frac{v_p(\mathcal{H})}{p} \right)
&=\prod_{\substack{p > H/q \\ p \nmid q}} \left( 1+O_k\left( \frac{1}{ p^2 } \right) \right)\\
&=1+O_k\bigg( \left( \frac{H}{q}\log \frac{H}{q} \right)^{-1} \bigg).
\end{align*}
On the other hand, since the first product of (\ref{eq:prod}) is $\ll_k \log^{k-1}(H/q)$ using Mertens' estimates, the lemma follows.  
\end{proof}

%With  at hand, we compute mixed moments of the number of primes in many short progressions.

\section{Proof of Theorem \ref{thm:ap}}

To lighten the notation, we adopt the convention $\widetilde{x}:=\frac{x}{\varphi(q)\log q}.$ Without loss of generality, we assume $I_i=(\alpha_i,\beta_i]$ with $0<\alpha_1<\beta_1<\cdots<\alpha_r<\beta_r$ and $k_i \geq 1$ for $i=1,\ldots, r.$ Then, adapting the proof of \cite[Theorem 1.2]{moments}, we have
\begin{gather*}
\frac{1}{\varphi(q)}\sideset{}{^*}\sum_{a \Mod{q}}
\prod_{i=1}^{r}  \#\left\{ p \equiv a \Mod{q} \,:\, 
\frac{p}{\varphi(q)\log q} \in I_i
\right\}^{k_i}  \\
=\frac{1}{\varphi(q)}\sideset{}{^*}\sum_{a \Mod{q}}
\prod_{i=1}^{r} 
\underset{\substack{ \widetilde{p_1^{(i)}}, \ldots, \widetilde{p_{k_i}^{(i)}} \in I_i \\ p_1^{(i)} \equiv \cdots \equiv p_{k_i}^{(i)} \equiv a \Mod{q} }}{\sum \cdots \sum} 1 \\
=\frac{1}{\varphi(q)}\sideset{}{^*}\sum_{a \Mod{q}} 
\prod_{i=1}^r \sum_{j_i=1}^{k_i} \stirling{k_i}{j_i} j_i ! 
\underset{\substack{ \widetilde{p_1^{(i)}}< \ldots< \widetilde{p_{j_i}^{(i)}} \in I_i \\ p_1^{(i)} \equiv \cdots \equiv p_{j_i}^{(i)} \equiv a \Mod{q} }}{\sum \cdots \sum} 1.
\end{gather*}
Making the change of variables $p_1^{(i)}=n_i+d_1^{(i)}, \ldots, p_{j_i}^{(i)}=n_i+d_{j_i}^{(i)}$ for $i=1,\ldots,r,$ this becomes
\begin{align*}
\frac{1}{\varphi(q)}&\sideset{}{^*}\sum_{a \Mod{q}}
\prod_{i=1}^r \sum_{j_i=1}^{k_i} \stirling{k_i}{j_i} j_i ! 
\underset{\substack{0=d_1^{(i)}<\cdots<d_{j_i}^{(i)} \\ 
d_1^{(i)} \equiv \cdots \equiv d_{j_i}^{(i)} \equiv 0 \Mod{q}}}{\sum \cdots \sum} \\
&\sum_{\substack{n_i \,:\, \widetilde{n_i+d_1^{(i)}}, \ldots, \widetilde{n_i+d_{j_i}^{(i)}} \in I_i \\ n_i \equiv a \Mod{q} }} 1_{\mathbb{P}}(n_i+d_1^{(i)}) \cdots 1_{\mathbb{P}}(n_i+d_{j_i}^{(i)}) \\
=\frac{1}{\varphi(q)}& \sideset{}{^*}\sum_{a \Mod{q}} \underset{1 \leq j_i \leq k_i, 1 \leq i \leq r}{\sum \cdots \sum}
\left( \prod_{i=1}^r \stirling{k_i}{j_i} j_i! \right)
\underset{\substack{0=d_1^{(i)}<\cdots<d_{j_i}^{(i)} \\ 
d_1^{(i)} \equiv \cdots \equiv d_{j_i}^{(i)} \equiv 0 \Mod{q}}}{\sum \cdots \sum} \\
& \underset{\substack{n_i \,:\, \widetilde{n_i+d_1^{(i)}}, \ldots, \widetilde{n_i+d_{j_i}^{(i)}} \in I_i \\ n_i \equiv a \Mod{q}}}{\sum \cdots \sum}
\prod_{i=1}^r 1_{\mathbb{P}}(n_i+d_1^{(i)}) \cdots 1_{\mathbb{P}}(n_i+d_{j_i}^{(i)}) \\
=
\frac{1}{\varphi(q)} & \underset{1 \leq j_i \leq k_i, 1 \leq i \leq r}{\sum \cdots \sum}
\left(\prod_{i=1}^r \stirling{k_i}{j_i} j_i! \right)
\underset{\substack{0=d_1^{(i)}<\cdots<d_{j_i}^{(i)} \\ 
d_1^{(i)} \equiv \cdots \equiv d_{j_i}^{(i)} \equiv 0 \Mod{q}}}{\sum \cdots \sum} \\
& \underset{\substack{n_i \,:\, \widetilde{n_i+d_1^{(i)}}, \ldots, \widetilde{n_i+d_{j_i}^{(i)}} \in I_i \\ n_1 \equiv \cdots \equiv n_k \Mod{q}\\ (n_1\cdots n_k,q)=1}}{\sum \cdots \sum}
\prod_{i=1}^r 1_{\mathbb{P}}(n_i+d_1^{(i)}) \cdots 1_{\mathbb{P}}(n_i+d_{j_i}^{(i)}).
\end{align*}
Making another change of variables $n_i=n+h_i$ for $i=1, \ldots, r,$ this becomes
\begin{gather}
\frac{1}{\varphi(q)}  \underset{1 \leq j_i \leq k_i, 1 \leq i \leq r}{\sum \cdots \sum}
\left(\prod_{i=1}^r \stirling{k_i}{j_i} j_i! \right)
\underset{\substack{0=d_1^{(i)}<\cdots<d_{j_i}^{(i)} \\ 
d_1^{(i)} \equiv \cdots \equiv d_{j_i}^{(i)} \equiv 0 \Mod{q}}}{\sum \cdots \sum}
\underset{\substack{0=h_1<\cdots<h_r\\h_1 \equiv \cdots \equiv h_r \equiv 0 \Mod{q}}}{\sum \cdots \sum} \nonumber\\ \label{eq:2ndcov}
\underset{\substack{n \,:\, \widetilde{n+h_i+d_1^{(i)}}, \ldots, \widetilde{n+h_i+d_{j_i}^{(i)}} \in I_i \\ (n,q)=1}}{\sum}
\prod_{i=1}^r 1_{\mathbb{P}}(n+h_i+d_1^{(i)}) \cdots 1_{\mathbb{P}}(n+h_i+d_{j_i}^{(i)}). 
\end{gather}
Since $ \widetilde{n+h_i+d_1^{(i)}}, \ldots, \widetilde{n+h_i+d_{j_i}^{(i)}} \in I_i$
for $i=1,\ldots,r$ if and only if 
\begin{align*}
\widetilde{n} \in \left( \max_{1 \leq i \leq r} (\alpha_i-\widetilde{h_i}), \min_{1 \leq i \leq r} (\beta_i-\widetilde{h_i}-\widetilde{d_{j_i}^{(i)}}) \right]=:I_{\boldsymbol{\alpha}, \boldsymbol{\beta};\boldsymbol{d}, \boldsymbol{h} },
\end{align*}
where $\boldsymbol{\alpha}:=(\alpha_1,\ldots,\alpha_r), \ldots, \boldsymbol{h}:=(h_1,\ldots,h_r),$ the expression (\ref{eq:2ndcov}) is $\Sigma_1+\Sigma_2,$ where
\begin{gather} 
\Sigma_1:= 
\frac{1}{\varphi(q)}  \underset{1 \leq j_i \leq k_i, 1 \leq i \leq r}{\sum \cdots \sum}
\left(\prod_{i=1}^r \stirling{k_i}{j_i} j_i! \right)
\underset{\substack{0=d_1^{(i)}<\cdots<d_{j_i}^{(i)} \\ 
d_1^{(i)} \equiv \cdots \equiv d_{j_i}^{(i)} \equiv 0 \Mod{q}}}{\sum \cdots \sum} \,
\underset{\substack{0=h_1<\cdots<h_r\\h_1 \equiv \cdots \equiv h_r \equiv 0 \Mod{q}\\
|I_{\boldsymbol{\alpha}, \boldsymbol{\beta};\boldsymbol{d}, \boldsymbol{h} }|>\widetilde{q}
}}{\sum \cdots \sum} \nonumber\\ 
\label{eq:selberg}
\underset{\substack{\widetilde{n} \in I_{\boldsymbol{\alpha}, \boldsymbol{\beta};\boldsymbol{d}, \boldsymbol{h} } \\ (n,q)=1}}{\sum}
\prod_{i=1}^r 1_{\mathbb{P}}(n+h_i+d_1^{(i)}) \cdots 1_{\mathbb{P}}(n+h_i+d_{j_i}^{(i)})
\end{gather}
and $\Sigma_2$ is the same sum but with the condition $|I_{\boldsymbol{\alpha}, \boldsymbol{\beta};\boldsymbol{d}, \boldsymbol{h} }|>\widetilde{q}$ replaced by $|I_{\boldsymbol{\alpha}, \boldsymbol{\beta};\boldsymbol{d}, \boldsymbol{h} }| \leq \widetilde{q}.$

%\begin{align*}
%|I_{\boldsymbol{\alpha}, \boldsymbol{\beta};\boldsymbol{d}, \boldsymbol{h} }|=\min_{1 \leq i \leq r} (\beta_i-\widetilde{h_i}-\widetilde{d_{j_i}^{(i)}})-\max_{1 \leq i \leq r} (\alpha_i-\widetilde{h_i}) \leq \widetilde{q}. 
%\end{align*}

Let us deal with $\Sigma_2$ first. Since $\widetilde{n} \in I_{\boldsymbol{\alpha}, \boldsymbol{\beta};\boldsymbol{d}, \boldsymbol{h} }$ and $|I_{\boldsymbol{\alpha}, \boldsymbol{\beta};\boldsymbol{d}, \boldsymbol{h} }| \leq \widetilde{q},$ by a simple application of the Selberg sieve (see \cite[Theorem 6.7]{MR2061214} for instance), the second line of (\ref{eq:selberg}) is
\begin{gather*}
\leq 
\sum_{\widetilde{n} \in I_{\boldsymbol{\alpha}, \boldsymbol{\beta};\boldsymbol{d}, \boldsymbol{h} }}
\prod_{i=1}^r 1_{\mathbb{P}}(n+h_i+d_1^{(i)}) \cdots 1_{\mathbb{P}}(n+h_i+d_{j_i}^{(i)}) \\
\ll_{\boldsymbol{k}} \mathfrak{S}(\{ h_i+d_1^{(i)}, \ldots, h_i+d_{j_i}^{(i)} \}_{i=1}^r) \times \frac{q}{(\log q)^{j_1+\cdots+j_r}},
\end{gather*}
where $\boldsymbol{k}=(k_1,\ldots,k_r)$. Therefore, it follows from Lemma \ref{lem:upperbound} that the sum $\Sigma_2$ is
\begin{gather*}
\ll_{\boldsymbol{k}} 
\frac{q}{\varphi(q)} \times \max_{j_1,\ldots,j_r} \frac{1}{(\log q)^{j_1+\cdots+j_r}}
\underset{\substack{0=\widetilde{d_1^{(i)}}<\cdots<\widetilde{d_{j_i}^{(i)}} <\beta_i-\alpha_i \\ 
d_{1}^{(i)} \equiv \cdots \equiv d_{j_i}^{(i)} \equiv 0 \Mod{q}}}{\sum \cdots \sum} \\
\underset{\substack{0=\widetilde{h_1}<\cdots<\widetilde{h_r}<\beta_r-\alpha_1\\h_1 \equiv \cdots \equiv h_r \equiv 0 \Mod{q}\\
|I_{\boldsymbol{\alpha}, \boldsymbol{\beta};\boldsymbol{d}, \boldsymbol{h} }| \leq \widetilde{q}
}}{\sum \cdots \sum} \mathfrak{S}(\{ h_i+d_1^{(i)}, \ldots, h_i+d_{j_i}^{(i)} \}_{i=1}^r) \\
\ll_{\boldsymbol{k}, \boldsymbol{\alpha}, \boldsymbol{\beta}} 
\frac{q}{\varphi(q)}\left(\log \frac{\varphi(q)\log q}{q} \right)^{k-1} 
 \max_{j_1,\ldots,j_r} \frac{1}{(\log q)^{j_1+\cdots+j_r}} \\
\underset{\substack{0=\widetilde{d_1^{(i)}}<\cdots<\widetilde{d_{j_i}^{(i)}}<\beta_i-\alpha_i \\ 
d_1^{(i)} \equiv \cdots \equiv d_{j_i}^{(i)} \equiv 0 \Mod{q}}}{\sum \cdots \sum} 
\underset{\substack{0=\widetilde{h_1}<\cdots<\widetilde{h_r}<\beta_r-\alpha_1\\h_1 \equiv \cdots \equiv h_r \equiv 0 \Mod{q}
}}{\sum \cdots \sum} 1 \\
\ll_{\boldsymbol{k}, \boldsymbol{\alpha}, \boldsymbol{\beta}} \frac{q}{\varphi(q)}\left(\log \frac{\varphi(q)\log q}{q} \right)^{k-1}
 \max_{j_1,\ldots,j_r} \frac{1}{(\log q)^{j_1+\cdots+j_r}} \times
 (\log q)^{(j_1-1)+\cdots+(j_r-1)+(r-1)} \\
\ll_{\boldsymbol{k}, \boldsymbol{\alpha}, \boldsymbol{\beta}} \frac{q}{\varphi(q)\log q}\left(\log \frac{\varphi(q)\log q}{q} \right)^{|\boldsymbol{k}|-1},
\end{gather*}
where $|\boldsymbol{k}|:=k_1+\cdots+k_r,$ which is negligible.
%Then since $d_2^{(i)} \equiv \cdots \equiv d_{j_i}^{(i)} \equiv 0 \Mod{q}$ for $i=1,\ldots, r$ and $h_2 \equiv \cdots \equiv h_r \equiv 0 \Mod{q},$ once $(j_1-1)+\cdots+(j_r-1)+(r-1)-1$ of the shifts $d$'s and $h$'s are chosen, the remaining one is uniquely determined.

It remains to estimate $\Sigma_1,$ which can be split into two sums $\Sigma_{1,1}$ and $\Sigma_{1,2},$ where
\begin{gather*}
\Sigma_{1,1}:=
\frac{1}{\varphi(q)}  \underset{1 \leq j_i \leq k_i, 1 \leq i \leq r}{\sum \cdots \sum}
\left(\prod_{i=1}^r \stirling{k_i}{j_i} j_i! \right)
\underset{\substack{0=d_1^{(i)}<\cdots<d_{j_i}^{(i)} \\ 
d_1^{(i)} \equiv \cdots \equiv d_{j_i}^{(i)} \equiv 0 \Mod{q}}}{\sum \cdots \sum} \,
\underset{\substack{0=h_1<\cdots<h_r\\h_1 \equiv \cdots \equiv h_r \equiv 0 \Mod{q}\\
|I_{\boldsymbol{\alpha}, \boldsymbol{\beta};\boldsymbol{d}, \boldsymbol{h} }| > \widetilde{q} \\
\{ h_i+d_1^{(i)}, \ldots, h_i+d_{j_i}^{(i)} \}_{i=1}^r \text{ admissible}
}}{\sum \cdots \sum} \nonumber\\ 
\sum_{\substack{\widetilde{n} \in I_{\boldsymbol{\alpha}, \boldsymbol{\beta};\boldsymbol{d}, \boldsymbol{h} } \\ (n,q)=1}}
\prod_{i=1}^r 1_{\mathbb{P}}(n+h_i+d_1^{(i)}) \cdots 1_{\mathbb{P}}(n+h_i+d_{j_i}^{(i)})
\end{gather*}
and $\Sigma_{1,2}$ is the sum of the remaining terms, i.e., the set $\{ h_i+d_1^{(i)}, \ldots, h_i+d_{j_i}^{(i)} \}_{i=1}^r$ is non-admissible.  

Suppose $\{ h_i+d_1^{(i)}, \ldots, h_i+d_{j_i}^{(i)} \}_{i=1}^r$ is non-admissible. Then
\begin{align*}
\sum_{\substack{\widetilde{n} \in I_{\boldsymbol{\alpha}, \boldsymbol{\beta};\boldsymbol{d}, \boldsymbol{h} } \\ (n,q)=1}}
\prod_{i=1}^r 1_{\mathbb{P}}(n+h_i+d_1^{(i)}) \cdots 1_{\mathbb{P}}(n+h_i+d_{j_i}^{(i)})
\leq |\boldsymbol{k}|,
\end{align*}
so that the sum $\Sigma_{1,2}$ is
\begin{gather*}
\leq 
\frac{1}{\varphi(q)}  \underset{1 \leq j_i \leq k_i, 1 \leq i \leq r}{\sum \cdots \sum}
\left(\prod_{i=1}^r \stirling{k_i}{j_i} j_i! \right)
\underset{\substack{0=\widetilde{d_1^{(i)}}<\cdots<\widetilde{d_{j_i}^{(i)}}<\beta_i-\alpha_i \\ 
d_1^{(i)} \equiv \cdots \equiv d_{j_i}^{(i)} \equiv 0 \Mod{q}}}{\sum \cdots \sum} \,
\underset{\substack{0=\widetilde{h_1}<\cdots<\widetilde{h_r}<\beta_r-\alpha_1\\h_1 \equiv \cdots \equiv h_r \equiv 0 \Mod{q}}}{\sum \cdots \sum} |\boldsymbol{k}| \\
\ll_{\boldsymbol{k}, \boldsymbol{\alpha}, \boldsymbol{\beta}}
\frac{1}{\varphi(q)} \left( \frac{\varphi(q)\log q}{q} \right)^{|\boldsymbol{k}|-1},
\end{gather*}
which is again negligible.

Finally, suppose $\{ h_i+d_1^{(i)}, \ldots, h_i+d_{j_i}^{(i)} \}_{i=1}^r$ is admissible. Then assuming HL[$q$] uniformly for $0 \leq h_i+d_1^{(i)},\dots, h_i+d_{k_i}^{(i)} \leq (\beta_r-\alpha_1) \varphi(q)\log q$, we have
\begin{align*}
&\sum_{\substack{\widetilde{n} \in I_{\boldsymbol{\alpha}, \boldsymbol{\beta};\boldsymbol{d}, \boldsymbol{h} } \\ (n,q)=1}}
\prod_{i=1}^r 1_{\mathbb{P}}(n+h_i+d_1^{(i)}) \cdots 1_{\mathbb{P}}(n+h_i+d_{j_i}^{(i)}) \\
=&
(1+o_{\boldsymbol{k}, \boldsymbol{\alpha},\boldsymbol{\beta}}(1))
\mathfrak{S}(\{ h_i+d_1^{(i)}, \ldots, h_i+d_{j_i}^{(i)} \}_{i=1}^r;q) \\
&\sum_{\substack{\widetilde{n} \in I_{\boldsymbol{\alpha}, \boldsymbol{\beta};\boldsymbol{d}, \boldsymbol{h} }\\(n,q)=1}} 
\left( \prod_{i=1}^r \log (n+h_i+d_1^{(i)}) \cdots \log (n+h_i+d_{j_i}^{(i)}) \right)^{-1},
\end{align*}
so that the sum $\Sigma_{1,1}$ is
\begin{gather*}
\frac{1}{\varphi(q)}  \underset{1 \leq j_i \leq k_i, 1 \leq i \leq r}{\sum \cdots \sum}
\left(\prod_{i=1}^r \stirling{k_i}{j_i} j_i! \right)
\underset{\substack{0=d_1^{(i)}<\cdots<d_{j_i}^{(i)} \\ 
d_1^{(i)} \equiv \cdots \equiv d_{j_i}^{(i)} \equiv 0 \Mod{q}}}{\sum \cdots \sum} \,
\underset{\substack{0=h_1<\cdots<h_r\\h_1 \equiv \cdots \equiv h_r \equiv 0 \Mod{q}\\ |I_{\boldsymbol{\alpha}, \boldsymbol{\beta};\boldsymbol{d}, \boldsymbol{h} }| > \widetilde{q} \\ \{ h_i+d_1^{(i)}, \ldots, h_i+d_{j_i}^{(i)} \}_{i=1}^r \text{ admissible}}}{\sum \cdots \sum} \\
\frac{1+o_{\boldsymbol{k}, \boldsymbol{\alpha},\boldsymbol{\beta}}(1)}{(\log q)^{j_1+\cdots+j_r}}
\underset{\substack{n \,:\, \widetilde{n+h_i+d_1^{(i)}}, \ldots, \widetilde{n+h_i+d_{j_i}^{(i)}} \in I_i \\ (n,q)=1}}{\sum \cdots \sum}
\mathfrak{S}(\{ h_i+d_1^{(i)}, \ldots, h_i+d_{j_i}^{(i)} \}_{i=1}^r;q).
\end{gather*}
Since the singular series $\mathfrak{S}(\{ h_i+d_1^{(i)}, \ldots, h_i+d_{j_i}^{(i)} \}_{i=1}^r;q)$ vanishes provided that the set $\{ h_i+d_1^{(i)}, \ldots, h_i+d_{j_i}^{(i)} \}_{i=1}^r$ is non-admissible, the condition on the admissibility can be disregarded. Also, the error term induced by dropping the condition $|I_{\boldsymbol{\alpha}, \boldsymbol{\beta};\boldsymbol{d}, \boldsymbol{h} }| > \widetilde{q}$ is negligible. Unravelling all the changes of variables made, we arrive at
\begin{align*}
&\frac{1+o_{\boldsymbol{k}, \boldsymbol{\alpha},\boldsymbol{\beta}}(1)}{\varphi(q)}\sideset{}{^*}\sum_{a \Mod{q}} 
\prod_{i=1}^r \sum_{j_i=1}^{k_i} \stirling{k_i}{j_i} 
\frac{j_i!}{(\log q)^{j_i}}
\underset{\substack{ \widetilde{n_1^{(i)}}< \ldots< \widetilde{n_{j_i}^{(i)}} \in I_i \\ n_1^{(i)} \equiv \cdots \equiv n_{j_i}^{(i)} \equiv a \Mod{q} }}{\sum \cdots \sum} \mathfrak{S}(\mathcal{N};q), 
%= &\frac{1+o_{\boldsymbol{k}, \boldsymbol{\alpha},\boldsymbol{\beta}}(1)}{\varphi(q)}
%\sideset{}{^*}\sum_{a \Mod{q}} 
%\prod_{i=1}^r \sum_{j_i=1}^{k_i} \stirling{k_i}{j_i} 
%\frac{1}{(\log q)^{j_i}}
%\underset{\substack{ \widetilde{n_1^{(i)}}, \ldots, \widetilde{n_{j_i}^{(i)}} \in I_i %\text{ \normalfont distinct} \\ n_1^{(i)} \equiv \cdots \equiv n_{j_i}^{(i)} \equiv a %\Mod{q} }}{\sum \cdots \sum} \mathfrak{S}(\mathcal{N};q),
\end{align*}
where $\mathcal{N}:=\{ n_1^{(i)}, \ldots, n_{j_i}^{(i)} \}_{i=1}^r.$ Applying Proposition \ref{prop:sing}, this is
\begin{gather*}
\frac{1+o_{\boldsymbol{k}, \boldsymbol{\alpha},\boldsymbol{\beta}}(1)}{\varphi(q)}
\sideset{}{^*}\sum_{a \Mod{q}}
\prod_{i=1}^r \sum_{j_i=1}^{k_i} \stirling{k_i}{j_i} 
\frac{1}{(\log q)^{j_i}} \times  \prod_{i=1}^r \left(\frac{|I_i|\varphi(q)\log q }{\varphi(q)} \right)^{j_i} \\
=(1+o_{\boldsymbol{k}, \boldsymbol{\alpha},\boldsymbol{\beta}}(1))
\prod_{i=1}^r \left(\sum_{j_i=1}^{k_i} \stirling{k_i}{j_i} |I_i|^{j_i}
\right)
\end{gather*}
and the theorem follows.

\section{Some statistics}

To illustrate Corollaries \ref{thm:shortint} and \ref{thm:ap}, i.e., the pseudorandomness of primes at large scales, we provide some statistics in this section.

Regarding primes in short intervals, we set $N=100000$ and let $n \in (N, 2N]$ be an integer. Then the following sequences of points represent the normalized primes $\{ \frac{p-n}{\log N} \in (0, 20] \,:\, p>n \}$ with $n=114159, 141971, 171693$ and $182097$ respectively. These values of $n$ are chosen by taking the early few digits of $\pi$ and adding $100000$ to each.

%\Floatbarrier
\begin{figure}[htp]
    \centering
    \includegraphics[scale=0.7]{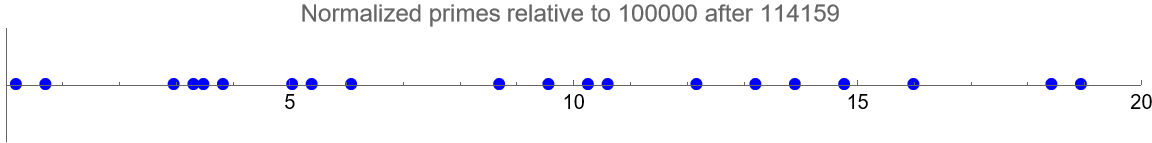}
  \\~\\
     \includegraphics[scale=0.7]{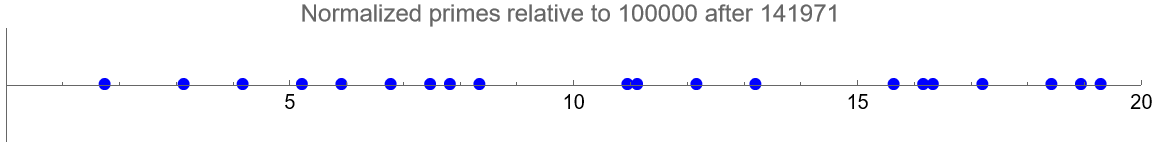}
     \\~\\
          \includegraphics[scale=0.7]{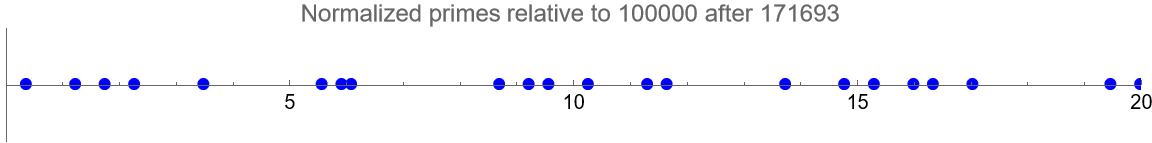}
     \\~\\
         \includegraphics[scale=0.7]{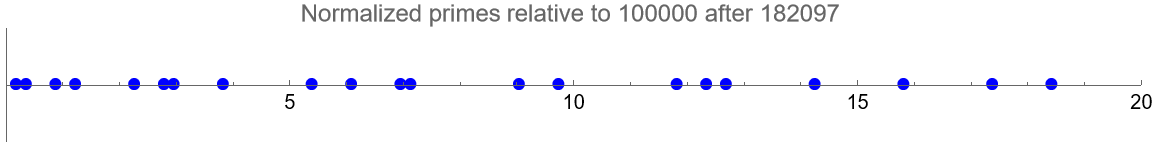}
    \caption{First few normalized primes in short intervals}
\end{figure}

Regarding primes in short progressions, we set $q=100000$ and let $a \Mod{q}$ be a reduced residue class. Then the following sequences of points represent the normalized primes $\{ \frac{p}{\varphi(q)\log q} \in (0, 20] \,:\, p \equiv a \Mod{q}  \}$ with $a \equiv 14159, 41971, 71693$ and $82097 \Mod{q}$ respectively. 

\begin{figure}[htp]
    \centering
    \includegraphics[scale=0.7]{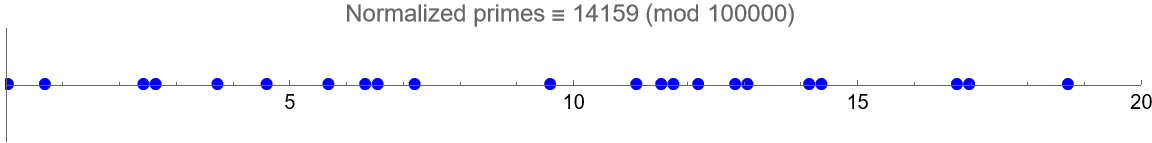}
  \\~\\
     \includegraphics[scale=0.7]{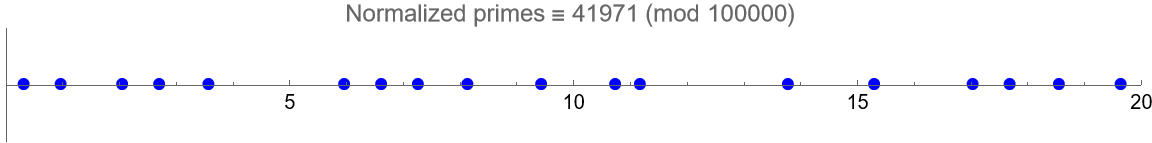}
     \\~\\
          \includegraphics[scale=0.7]{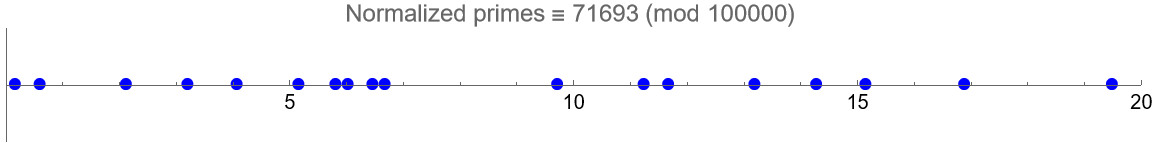}
     \\~\\
         \includegraphics[scale=0.7]{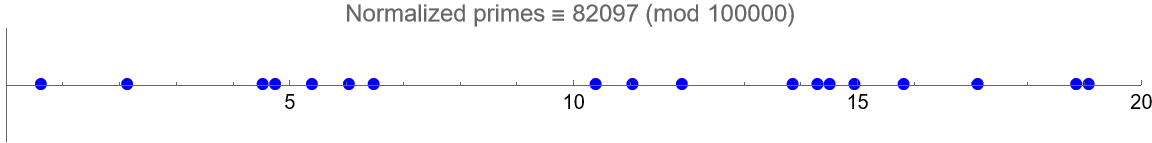}
    \caption{First few normalized primes in short progressions}
\end{figure}

%\section{Unconditional pseudorandomness of reduced residues}

\section*{Acknowledgements}
The author is grateful to Andrew Granville for his advice and encouragement. He would also like to thank Tony Haddad and Brad Rodgers for helpful discussions, Cihan Sabuncu and Christian T\'{a}fula for their careful proofreading, as well as the anonymous referee for valuable comments and corrections.

%\newpage
%\nocite{*}
\printbibliography

\end{document}